%% file: babylon7.tex
\documentclass[11pt]{article}
\usepackage{amssymb, amsmath, amsthm, epsf, epsfig}

\newcommand{\COLORON}{0}
\newcommand{\NOTESON}{0}
\newcommand{\Debug}{0}
\input{defs}

\newlength{\originalbase}
\newcommand{\Res}{{\cal R}}

\begin{document}

\newcommand{\comxy}{\ensuremath{ET^{x \leftrightarrow y}}}
\newcommand{\effrxy}{\ensuremath{\Res^{x y}}}
\newcommand{\cmfwxy}{\ensuremath{T^{x y}_{A\rightarrow}}}
\newcommand{\cmfwyx}{\ensuremath{T^{x y}_{A\leftarrow}}}
\newcommand{\cminxy}{\ensuremath{T^{x  y}_{A}}}
\newcommand{\cmaxxy}{\ensuremath{T^{x y}_{A\leftrightarrow}}}
\newcommand{\Ecmfwxy}{\ensuremath{\Ex \cmfwxy}}
\newcommand{\Ecmfwyx}{\ensuremath{\Ex \cmfwyx}}
\newcommand{\Ecminxy}{\ensuremath{\Ex \cminxy}}
\newcommand{\Ecmaxxy}{\ensuremath{\Ex \cmaxxy}}
\newcommand{\andcom}[1]{\ensuremath{\overleftrightarrow{#1}}-commute}

\title{Two-Color Babylon}

\author{Agelos Georgakopoulos\thanks{The first author acknowledges support by the University of Ottawa; georgakopoulos@tugraz.at}
\, and Peter Winkler\thanks{Research
supported by NSF grant DMS-0901475. Department of Mathematics, Dartmouth,
Hanover NH 03755-3551, USA; peter.winkler@dartmouth.edu. }}

\maketitle

\begin{abstract}
We solve the game of Babylon when played with chips of two colors,
giving a winning strategy for the second player in all previously
unsolved cases.
\end{abstract}

\section{Introduction}

The game of Babylon begins with $n$ chips of various colors, spread out
as ``stacks'' of height 1.  Players Alice and Bob alternate combining
two stacks which have the same number of chips or the same color
on top (or both).  The last player to make a legal move wins.

The game was introduced in 2003 by French designer Bruno Faidutti \cite{F} and is available
in commercial form, with 12 chips in four colors, three of each color.
(Of course, the game can be played with ordinary poker chips.)  Computer-aided
analysis \cite{B,G,M} has shown that with best play the aforementioned
initial configuration of the commercial game is a win for the second player.

The simplicity and attractiveness of Babylon has led researchers to try
to solve the game in general, and in particular the cases where the
tokens are of only two colors.  Goadrich and Schlatter \cite{GS} show that
the first player (Alice) can always win if there are only one or two chips of the
minority color, or if $n$ is odd; they conjectured that Bob can win in the
remaining, most difficult cases, when $n$ is even and there are three or
more chips of the minority color.  We prove their conjecture, completing
the winning-player classification for two-color Babylon.  Our strategy for
Bob is explicit and easy to implement in practice.

\section{General Observations}

Alice wins any game of Babylon if the total number of moves is odd;
Bob, if even.  Since the number of stacks decreases by one at each player's
turn, the number of moves is at most $n{-}1$.  In two-color Babylon, the
only other possibility is that there are $n{-}2$ moves, ending with two
stacks of different height and color (the color of a stack is naturally defined to be the color of its top chip; the colors of its other chips cannot affect the game any more).

The one-color game is of course trivial, with Alice winning when
$n$ is even and Bob when $n$ is odd.  If $n$ is even and there are only
one or two chips of the minority color (which, continuing the convention
of \cite{GS}, is red), Alice wins easily by obliterating the red chips;
if there's only one she covers it with a blue chip, and if there are two
she stacks them and at her next move covers them with the blue 2-stack
that Bob was forced to make at his first move.

If $n$ is odd and there are one or two red chips, Alice needs to keep them
alive; this is easily done by stacking red on blue (if there is one red chip)
or stacking the two red chips, then stacking the result on Bob's blue 2-stack.
Thereafter if Bob ever threatens the red stack by building a blue one of the
same height, Alice piles her red stack on top of the challenging blue one.
Thus the size of the red stack is always a power of 2, and at least 4.
There is never any danger that Alice will herself be forced to make a blue
stack of the same height as the red one, because this can only happen if all
blue stacks are exactly half the height of the red one; but that is impossible
because $n$ is odd.

Thus, Alice wins all two-color games that begin with one or two red chips.
When $n$ is odd and there are three or more red chips, Alice again needs
to preserve both colors; and again, she can easily do so, as shown in \cite{GS}.
Thus the critical remaining case is when there are three or more red chips
and $n$ is even.  This time it is Bob who needs to keep both colors alive.
The difficulty of this case, relative to the odd-$n$ case, is twofold: (1) Bob
lacks the power of first move, and (2) Bob must take care not to create, or permit
Alice to create, a stack of height $n/2$ that is the only stack of its color.

We will show that Bob can overcome these difficulties and win in all cases,
completing the classification for two-color Babylon.

\section{Notation}

We use the symbol $b$ to denote a blue chip, $r$ a red chip; $x$ is a chip of either
color and $y$ a chip of the {\em other} color; $z$ stands for a chip of any
color.  We will be using strings of the above symbols to describe stacks, with the left-most symbol standing for the chip
on top.  A {\em hill} is a stack of height at least two.  When there is no possibility
of confusion we often denote a hill by the capitalized form of the notation for
its color, thus we could write $R$ instead of $rbb$ and $X$ instead of $xxyz$.

In this notation a {\em move} by either player consists of appending one string (stack)
to another, always with the restriction that the two stacks were of the same color
or the same height.

Since the other cases are solved, we will be concerned exclusively with games in
which there are initially an even number $2m$ of chips, each its own stack, and
at least three of the minority color (red).  The parity of the number of stacks
at any state of the game tells us who is next to move; Alice moves next from any
even state, Bob from any odd state.

A state will be termed {\em safe} if it leads to a win by Bob, with best
play.  Equivalently, we may say that an even state in which Alice cannot move
(i.e., a two-stack state consisting of red and blue hills of different heights)
is safe, that an even state is safe if any move of Alice produces a safe state,
and an odd state is safe if there is some move by Bob leading to another safe state.
Our objective is to show that the initial state is safe.

\section{Main Results}

\begin{theorem} \label{thm:main}
Two-color Babylon with $p$ chips of one color and $q$ of the other, $1 \le p \le q$,
is a win for the second player if and only if $p{+}q$ is even and $p \ge 3$.
\end{theorem}

\begin{proof}
As noted in the introduction, owing to the results in \cite{GS} we need only provide
a winning strategy for Bob when $p{+}q=2m$ and $p \ge 3$.  We will show that Bob
can win by creating even states in which there is at most one hill of each color,
until an ``end game'' is reached in which he has to be a bit more careful.

A {\em singleton} is a stack that is not a hill, that is, one consisting of just one chip. 
A state with $j$ red singletons, $k$ blue singletons, 
red hills of heights $u_1,u_2,\dots$ and blue hills of heights $v_1,v_2,\dots$
will be denoted by $\langle j,k;u_1,u_2,\dots;v_1,v_2,\dots \rangle$.  The initial
configuration is thus $\langle p,q;; \rangle$.

Even states with exactly one hill of each color, each of even height, and at least four
singletons (stacks of height one) of each color, will be called {\em target} states.
Thus, a target state with $j$ red singletons, $k$ blue singletons, a red hill of height
$2u$ and a blue hill of height $2v$ will be denoted by $\langle j,k;2u;2v \rangle$.

Alice's options at a target state, and our notations for them, are:
\begin{enumerate}
\item place a singleton on top of another singleton: $xz$ 
\item place a singleton on top of the hill of its color (or vice-versa): $xX$
\item place one hill on top of the other: $XY$.
\end{enumerate}

Of course, option (2) requires that a hill exist, and (3) that there are two
hills and that they are of the same height.

Bob normally answers $xz$ with $xzX$, $xX$ with $xX$ and $XY$ with $yy$, always creating a new
even state in which there is again exactly one hill of each color, each of even height. We call this kind of play by Bob the {\em even-hill-strategy}.
In particular, $\langle j,k;u;v \rangle$ transforms to either $\langle j{-}2,k;2u{+}2;2v \rangle$,
$\langle j,k{-}2;2u;2v{+}2 \rangle$, $\langle j{-}1,k{-}1;2u{+}2;2v \rangle$, $\langle j{-}1,k{-}1;2u;2v{+}2 \rangle$,
$\langle j{-}2,k;2;2v{+}2u \rangle$, or $\langle j,k{-}2;2u{+}2v;2 \rangle$.  Since the number of singletons
of one or both colors may have been reduced below four, the new state might not be a target state.

Observe that if $p \ge 6$ then Bob can create a target state at his first move: if Alice
opens with $xx$ he replies with $yy$, and if she opens with $xy$ he replies with $yx$.
Either way the result is target state $\langle p{-}2,q{-}2;2,2 \rangle$.


Before proceeding further it will be useful to understand what happens when there
is only one stack left of one of the colors.

\begin{lemma} \label{lemma:1}
Any even state with just one stack $X$ of color $x$ and at least one of color $y$ is safe,
provided the height of $X$ exceeds $m$.
\end{lemma}

\begin{proof} Nothing can ever be combined with $X$, so both colors survive and Bob wins.
\end{proof}

\begin{lemma} \label{lemma:2}
Any odd state with just one stack $X$ of color $x$ is safe,
provided the height $u$ of $X$ is less than $m$, not all $y$-stacks are of height $u/2$,
and at most two are of height $u$.
\end{lemma}

\begin{proof}
We proceed by induction on the (even) number $t$ of $y$-stacks.  If there are two $y$-stacks of height $u$,
Bob stacks them; if one, he stacks it with any other $y$-stack.  All $y$-stacks of height $u$ are thus
killed.  In all other cases, Bob makes the largest $y$-stack possible of height $u' \not= u$.
(Since the $y$-stacks are not all of height $u/2$, either the tallest two or the shortest two
will combine to make a stack of height $\not=u$; thus $u'$ is well-defined.)  Now Alice cannot change
the height of $X$, nor can she suddenly make three $y$-stacks of height $u$.  Can she cause
all the $y$-stacks to become of height $u/2$?  No: if there had been a $y$-stack of height $u/2$
before we would now have $u' > u/2$.  So we are reduced to the cases where Alice has created a first
or second $y$-stack of height $u/2$.  Since Alice leaves an even number of $y$ stacks, and
$u + u/2 + u/2 < 2u < 2m = n$, there must be a stack of height different than $u/2$ after her move.
\end{proof}

Let us return to the proof of the main theorem.
Bob's strategy is to continue applying the even-hill-strategy as long as it leaves a target state for Alice, i.e., it leaves at least four singletons in each of the colors. Now consider the first position $S$ faced by Bob such that playing from $S$ according to the even-hill-strategy would leave less than four singletons in one of the colors. If following the even-hill-strategy would yield a position with 3 singletons in one color and at least 3 in the other color (which happens if $S=\langle 3,k;2u,2;2v \rangle$ or $S=\langle 3,k;2u;2v,2 \rangle$  with $k \ge 3$, Alice having just created a stack of height 2, or if $S=\langle 5,k;0;2v\rangle$, Alice having just played $BR$), then Bob wins by the following lemma.

\begin{lemma} \label{lemma:4}
Any even state of the form 
$\langle 3,k;2u;2v \rangle$ with $k\geq 3$ and $u+v\geq 3$ is safe.
\end{lemma}
\begin{proof}

\renewcommand{\labelenumi}{(\roman{enumi})}
We analyze all possible replies by Alice. In each case Bob must be careful not to allow Alice to create a solo
stack of height $m$.  His moves must thus take into account the value of $2u$
(of concern when at or just below $m$) and the relative values of $2u$ and $2v$
(of concern when they are both around $m/2$).

\begin{enumerate}
\item Alice replies $rb$;

By \Lr{lemma:3} Bob wins by $rbR$ unless $2u+4=m$ and $k>3$. If the latter is the case, then Alice threatens to create a single red stack containing half the chips. But Bob can reply with $br$, and it is now easy for him to make sure that both colors survive; see Lemmas \ref{lemma:1} and \ref{lemma:2}.

\item Alice replies $br$;

Arguing as in the previous case, we can play $brB$ unless $2u+2=m$ and $k>3$. Note that in the latter case $u>1$ must hold for $u+v\geq 3$. Bob can then play $br$, threatening $rR$ in his next move which would finish the game by \Lr{lemma:2}. Alice can try to prevent this by $rb$, but then $brrb$ wins for Bob  by \Lr{lemma:2} again.

\item Alice replies $rr$;

If $2u+3=m$ then $br$ is an easy win for Bob. If $u=v$ then $BR$ is easily seen to be safe using \Lr{lemma:2}. If neither of the above is the case, then $rrr$ is safe as it allows Bob to create a single red stack on his next move.

\item Alice replies $rR$;

If $2u+3=m$, or  $2u+3=m+2$, or $2u+3=m-2$ (the latter two possibilities are dangerous if and only if $v=1$) then $br$ leads an easy win for Bob. Otherwise $rr$ wins.

\item Alice replies $RB$;

If the height $H$ of the new hill is $m{-}3$ or $m{-}2$ then $br$ is safe. If it is $m{-}1$ then $rr$ is safe. The other cases are easy.

\item Alice replies $BR$;

Bob easily wins by $rr$ by making sure that red survives.

\item Alice replies $bB$ or $bb$;

We have $k>3$ \obda, for otherwise we can switch colors. Thus in both cases, $bbB$ creates a position of the same form as the original one, and we proceed by induction.
\end{enumerate}
\end{proof}

Thus, it only remains to consider positions $S$ from which the even-hill-strategy would yield a position with exactly 2 singletons in one color and $S$  has not been preceeded by a position with exactly 3 singletons in that color, for we have already seen that such positions are won by Bob. Since Alice's last state was a target state, such an $S$ must have one of the following forms (or its red-blue reflection):

\renewcommand{\labelenumi}{\arabic{enumi}.}
\begin{enumerate}
\item $\langle 3,k;2u{+}1;2v \rangle$ with $k \ge 4$ (if Alice just played $rR$);
\item $\langle 2,k;2u,2;2v \rangle$ with $k \ge 4$ and even (if Alice just played $rr$).
\item $\langle 4,k;0;2v \rangle$ with $k \ge 4$ and even (if Alice just played $BR$). 
\end{enumerate}
 
These cases will be easily solved using the following lemma.

\begin{lemma} \label{lemma:3}
Any even state of the form $\langle 2,2s;2u,2v \rangle$ with $u,v \geq 1$ is safe unless $2u+2=m$ and $s>1$.
\end{lemma}
\begin{proof}
We analyze all possible moves of Alice from this state:
\renewcommand{\labelenumi}{(\roman{enumi})}
\begin{enumerate}
\item Alice plays $BR$;
(Provided $u=v$.) Bob plays $rr$ and can make sure that the $rr$ stack survives to the end by immediately obliterating any blue stack of size 2 Alice creates.

\item Alice plays $RB$;
If the height $2u+2v$ of the new stack $X$ is greater than $m$ or less than $m-2$ then Bob plays $rr$ and wins easily by \Lr{lemma:1} or \Lr{lemma:2}.
Otherwise we have to be slightly more careful to avoid a red stack of height $m$: if $2u+2v=m$ play $rX$, if $2u+2v=m-1$ then play $rr$, and if $2u+2v=m-2$ play $br$ and counter $rb$ with $brrb$.

\item Alice plays $bb$;
We may assume $s=1$ \obda. If $2v+2\neq m$ then $bbB$ wins.  If $2v+2= m$, in which case $u=v$ holds, we win by $BR$.  

\item Alice plays $bB$;
We may assume $s=1$ \obda. If $2v+2\neq m$ then $bbB$ as above.  If $2v+2= m$ then $rb$ wins.

\item Alice plays $br$;

If $2u= m-1$, then $u>1$ because we have at least 8 chips in total. It is now easy to see that $br$ wins. If $2u\neq m-1$ then $rR$ wins by \Lr{lemma:1} or  \Lr{lemma:2}.

\item Alice plays $rb$;

Bob easily wins by playing $rrb$, unless $2u=m-3$ in which case he can safely play $br$.

\item Alice plays $rR$;

Bob easily wins by playing $rrb$, unless $2u=m-2$ in which case he can safely play $br$.

\item Alice plays $rr$;

This is the only case in which, under certain circumstances, Bob can lose.  Indeed, if $2u+2=m$, then Alice threatens to combine all red chips in a stack of height $m$,
winning the game. The only way for Bob to defend is to capture the red hill of height $2u$ with the blue hill, and he can do so \iff\ $u=v$ which implies $s=1$. 

If $2u+2\neq m$ then playing $rrR$ reaches a safe position by \Lr{lemma:1} or \Lr{lemma:2}.

\end{enumerate}
\end{proof}

We can now analyze the positions 1-3 above.

For position 1 note that, by \Lr{lemma:3}, $rR$ is safe unless $2u+4=m$. If the latter is the case, then $br$ will easily allow Bob to create a single red stack
of height different than $m$, winning by Lemmas \ref{lemma:1} and \ref{lemma:2}.

Position 2 is very similar and we leave the details to the reader.

Finally, in position 3 Bob can play $rr$ and win by \Lr{lemma:3} since $m>4$.
\medskip

This completes the proof in all cases when $p\geq 6$, and so Bob can create a target state at his first move as already mentioned. It remains to consider the simple cases when $p=3,4$ or $5$.

When $p=5$, Bob has to deviate from his general strategy only slightly. If Alice opens with $bb$ or $rb$ then he replies with $rb$ and $bb$ respectively.
This creates a target state unless the number of blue chips $q$ is also 5. But in the latter case \Lr{lemma:3} applies, with blue being the minority color now.
If Alice opens with $br$ then $rr$ wins by  \Lr{lemma:3} again. Finally, if Alice opens with $rr$ then $rr$ will quickly allow Bob to form a single red stack, winning by \Lr{lemma:1} or \Lr{lemma:2}.

When $p=4$ then Bob can play his first move according to his 2-hill strategy for the general case unless $q=4$; this immediately reaches a position $\langle 2,2s;2;2 \rangle$,
which is a win by \Lr{lemma:3}. If $q=4$ then Bob can force the position $\langle 1,3;2,2; \rangle$ after the first move, and it is easy to win that position.

Finaly, the case $p=3$ is an easy excersise left for the entertainment of the reader.

\end{proof}

\end{document}

%% file: defs.tex
\usepackage[usenames,dvips]{color} 
\usepackage{amsthm,amssymb,amsmath,enumerate,graphicx,epsf,mathbbol,stmaryrd}
\usepackage[dvips, bookmarks, colorlinks=false, breaklinks=true]{hyperref}

\renewcommand{\labelenumi}{\theenumi}

\newcommand{\comment}[1]{}
\newcommand{\COMMENT}[1]{}

\definecolor{darkgray}{rgb}{0.3,0.3,0.3}

\comment{
	\begin{lemma}\label{}	
\end{lemma}
\begin{proof}

\end{proof}

\begin{theorem}\label{}
\end{theorem} 
\begin{proof} 	

\end{proof}

}

\newtheorem{proposition}{Proposition}[section]

\newtheorem{theorem}[proposition]{Theorem}

\newtheorem{lemma}[proposition]{Lemma}

\newtheorem{examp}[proposition]{Example}

\newcommand{\FIG}{0}

\ifnum \NOTESON = 1 \newcommand{\note}[1]{ 

	\ 

	{\color{blue} \hspace*{-60pt} NOTE: \color{Turquoise}{\small  \tt \begin{minipage}[c]{1.1\textwidth}  #1 \end{minipage} \ignorespacesafterend }} 
	
	\ 
	
	}
\else \newcommand{\note}[1]{} \fi

\newcommand{\afsubm}[1]{ \ifnum \Debug = 1 {\mymargin{#1}}
\fi} 

\ifnum \Debug = 1 
\else  \fi

\ifnum \FIG = 1 
\else  \fi

\ifnum \FIG = 1 
\else  \fi

\ifnum \Debug = 1 \usepackage[notref,notcite]{showkeys}
\fi

\ifnum \COLORON = 0 \renewcommand{\color}[1]{}
\fi

\newcommand{\Ex}{\mathbb E}

\newcommand{\Lr}[1]{Lemma~\ref{#1}}

\renewcommand{\iff}{if and only if}

\newcommand{\obda}{without loss of generality}

\newcommand{\mymargin}[1]{
  \marginpar{%
    \begin{minipage}{\marginparwidth}\small%
      \begin{flushleft}%
        {\color{blue}#1}%
      \end{flushleft}%
   \end{minipage}%
  }%
}%

\newcommand{\mySection}[2]{}

